\documentclass[12pt]{amsart}
\usepackage{amssymb}
\usepackage{euscript}
\usepackage[english]{babel}

\theoremstyle{plain}
\newtheorem{theorem}{Theorem}
\newtheorem{lemma}[theorem]{Lemma}
\newtheorem{proposition}[theorem]{Proposition}

\newtheorem*{problem}{Problem}



\newcommand\surfs{\EuScript{S}}
\newcommand\var{\EuScript{V}}
\newcommand\hodge{\EuScript{H}}
\newcommand\R{\mathbb{R}}       
\newcommand\bH{\mathbb{H}}     
\newcommand\C{\mathbb{C}}      
\renewcommand\P{\mathbb{P}}       
\newcommand\Q{\mathbb{Q}}       
\newcommand\Hom{{\rm Hom}}

\begin{document}

\title[Non-geodesic variations of maximum dimension]{Non-geodesic variations of Hodge structure of maximum dimension}
\author[Carlson]{James A. Carlson}
\address{Department of Mathematics\\University of Utah\\Salt
Lake City, UT 84112}
\email{jxxcarlson@gmail.com}

\urladdr{http://www.math.utah.edu/\textasciitilde carlson}
\author[Toledo]{Domingo Toledo}
\address{Department of Mathematics\\University of Utah\\Salt
Lake City, UT 84112}
\email{toledo@math.utah.edu}
\urladdr{http://www.math.utah.edu/\textasciitilde toledo}
\date{\today}
\thanks{Second author supported by Simons Foundation Collaboration Grant 208853}
\begin{abstract}
There are a number of examples of variations of Hodge structure of maximum dimension.  However, to our knowledge, those that are global on the level of the period domain are totally geodesic subspaces that arise from an orbit of a subgroup of the group of the period domain. That is, they are defined by Lie theory rather than by algebraic geometry.  In this note, we give an example of a variation of maximum dimension which is nowhere tangent to a geodesic variation.  The period domain in question, which classifies weight two Hodge structures with $h^{2,0} = 2$ and $h^{1,1} = 28$, is of dimension $57$. The horizontal tangent bundle has codimension one, thus it is an example of a holomorphic contact structure, with local integral manifolds of dimension 28. The group of the period domain is $SO(4,28)$, and one can produce global integral manifolds as orbits of the action of subgroups isomorphic to $SU(2,14)$.  Our example is given by the variation of Hodge structure on the second cohomology of weighted projective hypersurfaces of degree $10$ in a weighted projective three-space with weights 
$1, 1, 2, 5$
\end{abstract}
\maketitle
\section{Introduction}
\label{sec:introduction}

\hyphenation{endo-morphisms}

Period domains $D = G/V$ for $G$ a (semi-simple, adjoint linear Lie group with a compact Cartan subgroup $T\subset G$ and $V$ the centralizer of a sub-torus of $T$) occur in many interesting situations.  It is known that there is a unique maximal compact subgroup $K\subset G$ containing $V$, 
so that there is a fibration 
\begin{equation}
\label{eq:fibration}
K/V \longrightarrow  G/V \buildrel \pi \over\longrightarrow  G/K
\end{equation}
of the homogeneous complex manifold $G/V$ onto the symmetric space $G/K$ with fiber the homogeneous projective variety $K/V$.  The tangent bundle $TD$ has a distinguished \emph{horizontal sub-bundle} $T_h D$ (also called the \emph{infinitesimal period relation}). It is a sub-bundle of the differential-geometric horizontal bundle (the orthogonal complement of the tangent bundle to the fibers). It usually, but not always a proper sub-bundle. When it is a proper sub-bundle, it is not integrable.  Typically, successive brackets of vector fields in $T_hD$ generate all of $ TD$.  We are interested in the case where the symmetric space $G/K$ is \emph{not} Hermitian symmetric.  In that case, the complex manifold $D$ admits invariant pseudo-K\"ahler metrics, but no invariant K\"ahler metric.

These manifolds were introduced by Griffiths as a category of manifolds that contains the classifying spaces  of Hodge structures.  For example, if $(H, \left< \ , \ \right>)$ is a real vector space of dimension $2p+q$ with a symmetric bilinear form of signature $2p,q$, 
the manifolds $SO(2p,q)/U(p)\times SO(q)$ classify Hodge decompositions  of weight two. Thus, we
have a direct sum decomposition
\begin{equation}
\label{ 
eq:hodgedecomp}
H^\C = H^{2,0}\oplus H^{1,1}\oplus H^{0,2}
\end{equation}
with Hodge numbers (dimensions)  $h^{2,0} = h^{0,2} = p $, $h^{1,1} = q$, and polarized by $\left< \ , \ \right>$: The real points of $H^{2,0}\oplus H^{0,2}$ form a maximal positive subspace, $H^{1,1}$ is the complexification of its  orthogonal complement  
(a maximal negative subspace), and so $(H^{2.0})^\perp = H^{2,0}\oplus H^{1,1}$.   Therefore the filtration
\begin{equation}
\label{eq:hodgefiltration}
H^{2,0}\subset (H^{2,0})^\perp \subset H^\C
\end{equation}
of $H^\C$ is the same as the Hodge filtration.  Therefore $H^{2,0}$ determines the Hodge filtration, hence the Hodge decomposition.  Note that $\left<u,\overline{v}\right>$ is a positive Hermitian inner product on $H^{2,0}$

The special orthogonal group of $\left< \ ,\ \right>$, isomorphic to $SO(2p,q)$, acts transitively on the choices of $H^{2,0}$, and the subgroup fixing one choice is isomorphic to $U(p)\times SO(q)$.
Thus, the homogeneous complex manifold $D = SO(2p,q)/U(p)\times SO(q)$ classifies polarized Hodge structures on a \emph{fixed} vector space $(H, \left< \ ,\  \right>)$.  Over $D$, there are tautological Hodge bundles $\hodge^{2,0},\hodge^{1,1},\hodge^{0,2}$.  The tangent bundle $TD$ and horizontal sub-bundle are 
\begin{equation}
\label{eq:hodgebundles}
TD = Hom_{\left< \ ,\ \right>}(\hodge^{2,0},\hodge^{1,1}\oplus \hodge^{0,2}),\ \ T_hD = Hom(\hodge^{2,0},\hodge^{1,1}),
\end{equation}
where $Hom_{\left< \ ,\ \right>}$ means homomorphisms $X$ which  preserving $\left< \  , \  \right>$ infinitesimally, that is, $\left<Xu,v\right> + \left< u,Xv\right> = 0$ for all $u,v \in H^{2,0}$.  If $X:H^{2,0}\to H^{1,1}$ this condition is vacuous, since $\left<H^{2,0},H^{1,1}\right> = 0$. Therefore $Hom_{\left< \ ,\ \right>}(\hodge^{2,0},\hodge^{1,1}) = Hom(\hodge^{2,0},\hodge^{1,1})$.

Whenever $p > 1$, the horizontal tangent bundle is a proper sub-bundle of the  tangent bundle.  The first interesting case is $p = 2$. If in addition $q = 2r$ is even, then the horizontal distribution locally a contact distribution, i.e., is the null space of a form $\omega = dz - (x_1 dy_1 + \cdots + x_r dy_r)$ in suitable local coordinates $(x,y,z)$.  Our example of weighted hypersurfaces yields a variation of Hodge structure of this type.

\subsection{Construction of horizontal maps}
\label{subsec:construction}

The  two main sources of horizontal holomorphic maps to period domains are
\begin{itemize}
  \item \emph{Totally geodesic maps}:  these come from Lie group theory, as orbits of suitable Lie subgroups of $G$.  For example, for the domains $SO(2p,2q)/U(p)\times SO(2q)$, we can put a complex structure $J$  on the underlying $\R$-vector space $H$, compatible with $< \ , \ >$.  Let $H^+, H^-$ denote the underlying real spaces of $H^{2,0}\oplus H^{0,2}$ and $H^{1,1}$ respectively. Consider the variation in which all $H^+$ are $J$-invariant.  This gives an embedding
  
   \begin{equation}
\label{ }
SU(p,q)/S(U(p)\times U(q)) \buildrel F \over\longrightarrow  SO(2p,2q)/U(p) \nonumber\times SO(2q)
\end{equation}

of the Hermitian symmetric space $D_1$ for $SU(p,q)$ in the domain $D$.     Since $H^+$ always remains $J$-invariant, the tangent vector to its motion, an element of $Hom(H^+,H^-)$ commutes with $J$.  Let $V\subset H^{1,1}$ be the space of $(1,0)$-vectors for $J$, that is, $V = \{X - iJX \ | X\in H^{1,1}\}$.   Then 
\begin{equation}
\label{ }
dF:TD_1 \to Hom(H^{2,0},V) \subset Hom(H^{2,0},H^{1,1} )= T_hD \nonumber
\end{equation}
in particular $F$ is horizontal and holomorphic.
\vskip .3cm
 \item \emph{Periods of  families of algebraic varieties}  This may be called the geometric method.  We proceed to explain it by describing the special case of  $SO(2p,2q)$:
 
\end{itemize}

 Let $\surfs \to B$ be smooth algebraic family of smooth projective algebraic surfaces over a smooth connected algebraic base $B$, fix a base point $b_0\in B$, and fix $(H, \left<\ , \ \right>)$ to be the pair ($H^2(\surfs_{b_0},\R)_{prim}$, intersection form).  For any $b\in B$ and a path $\lambda$ from $b_0$ to $b$, there is an isomorphism $\lambda^\#:H^2(\surfs_b)\to H^2(\surfs_{b_0})$, where different paths give different isomorphisms related by an element of the image of the monodromy representation $\rho:\pi_1(B,b_0) \to Aut(H^2_{prim}(\surfs_{b_0}))$. The \emph{period map} $F$  is defined by the rule:  $F(b)$ is the Hodge structure $\lambda^\#$(Hodge structure on $H^2(\surfs_b)$).  In this way, $F(b)$ is a Hodge structure on a fixed vector space, hence an element of $D$, well defined up to the action of the monodromy group.  We could look at this as a function of $b$ and $\lambda$, in which case we are lifting $F$ to a map $\widetilde{F}$ on a covering space of $B$.   Thus we have two equivalent formulations $F, \widetilde{F}$ of the period map related as follows: 
  \begin{eqnarray}
\begin{array}{ccc}
\label{eq:periodmaps}
    \widetilde{B} &  \buildrel \widetilde{F} \over\longrightarrow  & D  \\
{p}\Big\downarrow & & \Big\downarrow   \\
B & \buildrel F\over \longrightarrow & \Gamma\backslash D
\end{array}
\end{eqnarray}
where $p:\widetilde{B}\to B$ is the covering corresponding to the kernel of $\rho$ and  $\Gamma$ is a suitable monodromy group (containing the image of $\rho$).     Locally, the two maps $F, \widetilde{F}$ are the same, except when $F(b)$ is fixed by some non-identity element of $\Gamma$.

Griffiths showed that \emph{$F$ is holomorphic and horizontal}, in other words, $d\widetilde{F}:T\widetilde{B}\to F^*T_h D \subset T D$.  Under suitable assumptions, the closure  $\overline{F(B)}$ is an analytic subvariety of $\Gamma\backslash D$, hence is a closed \emph{horizontal analytic subvariety} of $\Gamma\backslash D$.

\subsection{A concrete example}

The preceding discussion can be applied to the family of smooth hypersurfaces in $\P^3$ of a fixed degree $d$.  In order to get non-constant variations and for the period domain not to be Hermitan symmetric we need to take $d\ge 5$.  

For $d=5$ we have that the Hodge numbers are $(4,44,4)$, hence $D = SO(8,44) / U(4)\times SO(44)$ has dimension $182$, the horizontal tangent space has dimension $176$ and the maximum dimension of an integral submanifold is $88$, the dimension of the horizontal $SU(4,22)$ orbit, see \cite{carlson}

We therefore find two horizontal maps:
\begin{itemize}
  \item Horizontal $SU(4,22)$ orbits of maximum dimension $88$.
  \item Periods of quintic surfaces, a \emph{maximal} integral manifold, see \cite{carlsondonagi} of dimension $40$ (the dimension of the moduli space of quintic surfaces).
\end{itemize}

In general, period domains, can have maximal integral manifolds of many different dimensions. Hypersurfaces generally yield integral manifolds of rather small dimension compared to the the maximum possible.  We would like to see geometric  examples of maximum, or close to maximum, dimension that come from geometry as opposed to Lie theory.   Hypersurfaces in weighted projective spaces provide such examples.

\section{The example}
\label{sec:example}

Let us consider the weighted projective space $\P(1,1,2,5)$ with coordinates $x_1,x_2,x_3,x_4$ with weights $1,1,2,5$ respectively.  One may think of $\P(1,1,2,5)$ as the quotient of $\C^4$ by the $\C^*$-action  $\lambda\in \C^*$ which acts by
\begin{equation}
\label{eq:weightaction}
\lambda \cdot (x_1,x_2,x_3,x_4) \longrightarrow (\lambda x_1,\lambda x_2, \lambda^2 x_3, \lambda^5 x_4)
\end{equation}
A weighted homogeneous polynomial of degree $d$  is a linear combination of monomials 
\begin{equation}
\label{eq:monomials}
x_1^{k_1} x_2^{k_2} x_3^{k_3}x_4^{k_4} \text{ of total weighted degree } d = k_1 + k_2 + 2 k_3 + 5 k_4
\end{equation}

For fixed $d$, the collection of weighted polynomials of degree $d$ forms a vector space that we will denote $S_d(1,1,2,5)$, or, simply $S_d$.   The direct sum $S(1,1,2,5) = \oplus_d S_d(1,1,2,5)$ is the algebra of weighted homogeneous polynomials. 

Any $f\in S_d$ defines a subvariety $V_f\subset P(1,1,2,5)$, namely $V_f = \{(x_1:x_2:x_3:x_4) | f(x_1,x_2,x_3,x_4) = 0 \}$.  If the only common solution of 
\begin{equation}
\label{ }
\frac{\partial f}{\partial x_1} = 0,\dots, \frac{\partial f}{\partial x_4} =0  \nonumber
\end{equation}
is  $(0,0,0,0)$, then $V_f$ is called a \emph{quasi-smooth} subvariety.  It is smooth except possibly for quotient singularities. Topologically it is a rational homology manifold, and in particular satisfies Poincar\'e duality over $\Q$.  Its second cohomology has a pure Hodge structure of weight two, polarized by the intersection form.

Fix $d$ and let $S_d^0\subset S_d$ denote the set, possibly empty, of all $f\in S_d$ for which $V_f$ is quasi-smooth.  For example, if $f\in S_4$, then no monomial in  $f$ can contain the variable $x_4$ of weight $5$, so $\frac{\partial f}{\partial x_4} = 0$ for all $f\in S_4$.  Therefore $S_4^0 =\emptyset$ since $(0:0:0:1)$ is a singular point of all $f\in S_4$.  On the other hand, a polynomial  in $S_d$ is a sum of powers of \emph{all}  of the variables defines a Fermat hypersurface.   These are always quasi-smooth.  In our case, one has the Fermat surface

\begin{equation}
\label{eq:fermat}
f_0 (x_1,x_2,x_3,x_4) = x_1^{10} + x_2^{10}  + x_3^5 + x_4^2 \in S_{10}^0,
\end{equation}

It has a rich structure, and, in particular, is double cover of  the 2-dimensional weighted projective plane with weights $1, 1, 2$, branched over a curve of degree ten.

The complement $\Delta_d = S_d \setminus S_d^0$ is a subvariety of $S_d$.  It is a proper subvariety if $S_d^0\ne \emptyset$.

Assume $S_d^0\ne\emptyset $.  Then $\Delta_d$ has complex codimension $1$ in $S_d$. Consequently, $S_d^0$ is connnected and we obtain a topologically locally trivial fibration $\var\to S_d^0$ where the fiber over $f$ is the variety $V_f$:
\begin{eqnarray}
\label{eq:universalfamily}
\begin{array}{ccc}
\var = \{(f,x) | f(x) = 0\} & \subset  & S_d^0 \times \P(1,1,2,5) \\
\Big\downarrow& &\Big\downarrow \\
S_d^0  & = & S_d^0
\end{array}
\end{eqnarray}

Fix a base point  $f_0\in S_d^0$.  Then there  is a monodromy representation $\rho:\pi_1(S_d^0,f_0) \to Aut(H^2(V_{f_0}))$, where $Aut$ is the group of automorphisms respecting all topological structures, in particular, the intersection form.  As $f$ varies, we transport the Hodge structure on $H^2(V_f,\C)_{prim} = H^{2,0}(V_f)\oplus H^{1,1}(V_f)_{prim} \oplus H^{0,2} (V_f)$ to $H^2(V_{f_0})_{prim}$, as explained in \S \ref{sec:introduction}, thus obtaining a point $F(f)\in D$, well defined up to the action of the image of $\rho$, where $D$ is the classifying space of Hodge structures on $H^2(V_{f_0})_{prim}$.   This defines   holomorphic period maps as in (\ref{eq:periodmaps}), namely
\begin{eqnarray}
\begin{array}{ccc}
\label{eq:universalperiodmaps}
\widetilde{S_{d}^0} &  \buildrel \widetilde{F} \over\longrightarrow  & D  \\
{p}\Big\downarrow & & \Big\downarrow   \\
S_{d}^0 & \buildrel F\over \longrightarrow & \Gamma\backslash D
\end{array}
\end{eqnarray}
where $\Gamma$ denotes the image of the monodromy representation $\rho$, and 
which is \emph{horizontal} in the sense that 
\begin{equation}
\label{ }
d\widetilde{F} : T \widetilde{S}_d^0 \longrightarrow  \widetilde{F}^* T_hD.
\end{equation}

We  must look carefully at some local properties of the period map $F$.   Let $U$ be a simply connected neighborhood of the base point $f_0$.  The inverse image of $U$ in $\widetilde{S^0_d}$
is a disjoint union of open sets isomorphic to $U$.  On such a  component of the inverse image, we can replace the  map $\widetilde{F}$ of (\ref{eq:universalperiodmaps}) by its restriction to a that connected component.  Identifying it with $U$, we may replace (\ref{eq:universalperiodmaps}) by the simpler diagram
\begin{eqnarray}
\begin{array}{ccc}
\label{eq:localuniversalperiodmaps}
&  & D  \\
 &\buildrel \widetilde{F}  \over \nearrow & \Big\downarrow   \\
U & \buildrel F\over \longrightarrow & \Gamma\backslash D
\end{array}
\end{eqnarray} 
Thus the period map $F$ to $\Gamma\backslash D$ is \emph{locally liftable} to $D$.  This is only an issue in the presence of fixed points.

Our example of a horizontal non-geodesic $V\subset \Gamma\backslash D$ will be $\overline{F(S_d^0)}$, the closure of the image of $F$, for suitable $d$.   We proceed to the necessary computations.

\subsection{The Jacobian Ring}
\label{subsec:jacobianring}

First of all,  choose $d = 10$, and  consider the space $S_{10}(1,1,2,5)$ of weighted homogeneous polynomials of degree $10$ with weights $(1,1,2,5)$.  Some computer experimentation led us to this choice. As noted above, the \lq\lq Fermat hypersurface\rq\rq  $V_{f_0}$ is defined by an element of $S_{10}$, and so  $S_{10}^0 \ne\emptyset$.  
Given $f\in S_{10}^0$, let 

\begin{enumerate}
  \item $J(f)\subset S$ denote the \emph{Jacobian ideal of $f$}, namely the ideal generated by the partial derivatives of $f$.
  \item $R(f) = S/J(f)$ be the \emph{Jacobian ring} of $f$.
  
\end{enumerate}

The Hodge decomposition and the differential of the period map have very explicit descriptions in terms of the graded ring $R(f)$ for $f\in S_{10}^0$.  Since the dimensions of the graded components $R_k(f)$  are independent of $f$, we often write simply $R_k$ for $R_k(f)$.
\begin{proposition} 
\label{prop:jacobiancomputations}
Let $f\in S_{10}^0$ and let $J$ and $R$ be as just defined.  Then
\begin{enumerate}
  \item $R_1 \cong H^{2,0}$
  \item $R_{11} \cong H^{1,1}$
  \item $R_{21}\cong H^{0,2}$
  \item $R_{22}\cong \C$ 
  \item $R_k = 0$ for $k>22$
  \item For $0\le i \le 22$, the pairing $R_i\otimes R_{22-i}\to R_{22}$ is non-degenerate.
  \end{enumerate}
  \end{proposition}
  
  \begin{proof}
 Statements of this type for projective hypersurfaces are consequences of the Griffiths residue calculus.  The analogous statements  for weighted projective hypersurfaces are proved in Theorem 1 of \cite{steenbrink} and in   \S4.3 of  \cite{dolgachev}.
  \end{proof}

 Applying the above to our situation, and using the polynomial $f_0$ to do computations,
 we find
 \begin{lemma}
 \label{lem:dimensions}
 \begin{enumerate}
  \item $h^{2,0} =2$, $h^{1,1} = 28$, $h^{0,2} = 2$
  \item $D = SO(4,28) /U(2) \times SO(28)$
  \item $D$ has dimension $57$.
 \item  The horizontal sub-bundle $T_h D = Hom(\hodge^{2,0},\hodge^{1,1})$ has fiber dimension $56$, hence is a holomorphic contact structure on $D$.
 \end{enumerate} 
 \end{lemma} 
 
 \begin{proof}
 Since the Hodge numbers are independent of $f$, we can compute them for $f_0$.  Using  Proposition \ref{prop:jacobiancomputations}, this is the same as computing the spaces $R_k(f_0)$, which amounts to a straightforward exercise of counting monomials.  First of all, $J$ is the ideal generated by $x_1^9,x_2^9, x_3^4,x_4$.  We find that
 \begin{enumerate}
  \item $R_1 = S_1 = \left< x_1,x_2 \right> $ is the vector space with basis $x_1,x_2$, so that  $h^{2,0} = h^{0,2} = 2$.
  \item $R_{11}$: to find a basis for this space, list all monomials that do not contain any of the above generators of $J$.  In particular, $x_4$ does not appear, so a basis consists of monomials in $x_1,x_2,x_3$ that do not contain $x_1^9, x_2^9,x_3^4$.  These can be conveniently grouped by powers of $x_3$:
  \begin{enumerate}
  \item $G_3 = \left< x_1^ix_2^{5-i}x_3^3 | i = 0,\dots 5 \right>$ is six-dimensional
  \item $G_2 = \left< x_1^ix_2^{7-i}x_3^2 | i = 0,\dots 5 \right>$ is eight-dimensional
  \item $G_1 = \left< x_1^ix_2^{9-i}x_3 | i = 1,\dots 8 \right>$ is eight-dimensional
  \item $G_0 = \left< x_1^ix_2^{11-i} | i = 3,\dots 8 \right>$  is six-dimensional
\end{enumerate}
\vskip .2cm
Therefore $\dim R_{11} = h^{1,1} = 28$
  \item It follows that $D$ classifies polarized Hodge structures with Hodge numbers $2,28,2$.  From the discussion in the introduction, it follows that $D = SO(4,28)/U(2)\times SO(28)$, which has dimension $57$ and  its  sub-bundle $T_hD = Hom(\hodge^{2,0},\hodge^{1,1})$ has fiber dimension $h^{2,0}h^{1,1} = 56$.  The easiest way to visualize $D$, and to see its dimension and the structure of the horizontal sub-bundle,  is to use its fibration (\ref{eq:fibration}) over the symmetric space.  In this case the symmetric space has real dimension $4\cdot 28$ and the fiber  is a projective line:
    \begin{eqnarray}
    \label{eq:fibrationD}
  \begin{array}{ccc}
SO(4)/U(2) & \longrightarrow & SO(4,28)/ U(2)\times SO(28) \\
& & \Big\downarrow {\pi}      \\
& &  SO(4,28) / S(O(4)\times O(28))
\end{array}
\end{eqnarray}
It is easy to see that $d\pi$ maps the fibers of $T_h D$ isomorphically (as real vector spaces) to the tangent spaces to the symmetric space. Thus $T_hD$ coincides, in this case, with the differential-geometric horizontal bundle.

\item To see that $T_hD$ is a holomorphic contact structure, recall the identification (\ref{eq:hodgebundles}), $TD \cong Hom_{\left< \ , \  \right>}(\hodge^{2,0},\hodge^{1,1}\oplus\hodge^{0,2})$.  Under this identification, $T_hD$ is identified with $Hom(\hodge^{2,0},\hodge^{1,1})$ as the kernel of the projection to $Hom_{< \ , \  >}(\hodge^{2,0},\hodge^{0,2})$.   Since 
$Hom_{\left< \ , \  \right>}(H^{2,0},H^{0,2})$ is a space of skew-symmetric en\-do\-mor\-phisms,  and since  $\dim H^{2,0}$  $= 2$, 
we see that  $$\dim Hom_{\left< \ ,\ \right>}(H^{2,0},H^{0,2}) = 1$$ The projection is a one-form $\omega$ with values in the line bundle  $T_vD = Hom_{\left< \ , \ \right>}(H^{2,0},H^{0,2})$ whose kernel is $T_hD$.  Here $T_vD$ stands for the vertical bundle. To be a contact structure means that it is totally non-integrable.  This means the following: if $X,Y$ are horizontal vector fields, then, for all $p\in D$,  $\omega([X,Y])_p$ depends only on $X_p,Y_p$, hence defines a bundle map $\Lambda^2 T_hD\to T_vD$.  To be a contact structure then means that this is a non-degenerate pairing. In other words, the resulting map $T_h D\to Hom(T_h D, T_v D)$ is an isomorphism.  This is a reformulation of the local coordinate condition $\omega\wedge (d\omega)^{28}\ne 0$ at every point.
 
Under our identification $T_h D \cong Hom(\hodge^{2,0},\hodge^{1,1})$, it is easy to check that  $\omega([X,Y])= X^t Y - Y^t X$, where the transpose is with respect to $< \ , \ >$, see \S 6 of \cite{carlsontoledotrans} for details.  One easily checks  that this paring is non-degenerate, so that we indeed have a contact structure.
\end{enumerate} 
 \end{proof} 
 
 Next, we  compute $dF$, where $F:S_{10}^0\to \Gamma\backslash D$ is the period map of (\ref{eq:universalperiodmaps}). The  group $G(1,1,2,5)$ of automorphisms of $P(1,1,2,5)$ acts on $S_{10}^0$ and  $F$ is  constant on orbits, so it should factor through  an appropriate quotient.  Since the group is not reductive, we avoid the technicalities of forming quotients, by working mostly on the  infinitesimal level.
 
 Given $f\in S_{10}^0$, the tangent space at $f$ to its $G(1,1,2,5)$-orbit is $J_{10}(f)$.  When we have a quotient, $R_{10}(f)$ can be identified with the tangent space to the quotient at the orbit of $f$.   We  use this fact as a guiding principle, relying on the fact that $d_fF$ vanishes on $J_{10}(f)$ and hence factors through $R_{10}(f)$. Thus we avoid working with the quotient directly.

To be more precise,  fix $f\in S_{10}^0$ and a simply connected neighborhood $U$ of $f$.  Since $\Gamma\backslash D$ need not be a manifold (and will not be at points fixed by non-identity elements of $\Gamma$), what we actually want to compute is $d_f\widetilde{F}$, where $\widetilde{F}:U\to D$ is a local lift of $F$ as in (\ref{eq:localuniversalperiodmaps}).

Since $U$ is an open subset of the vector space $S_{10}$, there is a canonical identification 
\begin{equation}
\label{eq:identifytangents}
T_fU \cong S_{10} \ \text{ by translation. }
\end{equation}
Under this identification, $J_{10}(f)$ is the tangent space to the orbit of $f$. Consequently, $d_f\widetilde{F}:S_{10}\to T_hD$ vanishes on  $J_{10}(f)$, hence factors through $R_{10}(f)$. 
Keeping in mind the  exact sequence
 \begin{equation}
\label{eq:exactsequence}
0\longrightarrow J_{10}(f) \longrightarrow S_{10}\buildrel p\over\longrightarrow R_{10}(f)\longrightarrow 0,
\end{equation}
we can state the main tool for computing differentials of period maps:

\begin{proposition}
\label{prop:multiplication}
Under the isomorphisms of Proposition \ref{prop:jacobiancomputations}, the isomorphism (\ref{eq:identifytangents}), and $p$ as in (\ref{eq:exactsequence}),  we have a commutative diagram
\begin{eqnarray}
\begin{array}{ccccc}
T_f U  &  &\buildrel d_f\widetilde{F} \over \longrightarrow & & T_hD \cong Hom(H^{2,0},H^{1,1})\\
{\cong}\Big\downarrow & & & & \Big\downarrow {\cong} \\
S_{10} &\buildrel p \over \longrightarrow & R_{10}(f)  & \buildrel m \over \longrightarrow & Hom(R_1(f),R_{11}(f))
\end{array}
\end{eqnarray}
where, for $\phi\in R_{10}$,  $m(\phi):R_1\to R_{11} $ is multiplication by $\phi$: if $x\in R_1$, then $m(\phi)(x) = \phi x$ 

\end{proposition} 

\begin{proof} This is the content of the residue calculus.  The isomorphisms between holomorphic objects and elements of the Jacobian ring preserve all natural products and pairings.
\end{proof}

The above proposition will allow us to compute the rank of $d\widetilde{F}$ at the point $f_0$ of (\ref{eq:fermat}).  We remark that, up to this point, the residue calculus and the corresponding algebraic facts about the Jacobian ring have closely paralleled the projective case.  But the failure of Macauley's theorem in the weighted projective case forces us to look carefully at the remaining statements.   Most results in the literature require assumptions on the weights, and on the degree, that are not satisfied for degree $10$ and weights $(1,1,2,5)$.  See the introduction and \S 1 of \cite{donagitu} for a general discussion of the possible difficulties that can appear in the weighted case.
\begin{proposition}
\label{prop:rank}
\begin{enumerate}
\item The rank of $d\widetilde{F}$ at $f_0$ is $28$, which is the maximum possible rank of a horizontal holomorphic map.
\item Let $W\subset T_h D$ denote the image of $d\widetilde{F}$.
Under the identification $T_hD \cong Hom(H^{2,0},H^{1,1})$, we have:
\begin{enumerate}
  \item For each $v\in H^{2,0}$, the subspace $Wv =_{def} \{Xv\  | \  X\in W\}\subset H^{1,1}$ has dimension $26$.
  \item $\{Xv \ | \  v\in H^{2,0}, X\in W\} = H^{1,1}$ 
\end{enumerate}

\end{enumerate}
\end{proposition}

\begin{proof}
By Proposition \ref{prop:multiplication} we need to compute the multiplication map $R_{10}\to Hom(R_1,R_{11})$.  In the proof of Lemma \ref{lem:dimensions} we found  a basis for $R_{11}$, and we can do a similar calculation with $R_{10}$:  a basis will be given by the monomials $x_1^a,x_2^b,x_3^c$ of total weight $10$ with $0\le a,b \le 8$ and $0\le c \le 3$.  These can again be conveniently grouped by the powers of $x_3$:

 \begin{enumerate}
  \item $G_3' =\left< x_1^ix_2^{4-i}x_3^3 | i = 0,\dots 5 \right>$ is five-dimensional
  \item $G_2' =\left< x_1^ix_2^{6-i}x_3^2 | i = 0,\dots 5 \right>$ is seven-dimensional
  \item $G_1' =\left< x_1^ix_2^{8-i}x_3 | i = 1,\dots 8 \right>$ is nine-dimensional
  \item $G_0' =\left< x_1^ix_2^{10-i} | i = 2,\dots 8 \right>$  is seven-dimensional
\end{enumerate}
\vskip .2cm
Therefore $\dim R_{10} = 28$, as claimed.

Next, we examine the map $m:R_{10}\to Hom(R_1,R_{11})$, where $m(\phi)$ is the homomorphism $m(\phi)(x) = \phi x$.  We claim that $m$  is injective.  Since $R_1=\left<x_1,x_2\right>$, it suffices to show that if $\phi\in R_{10}$ and both $\phi x_1 = \phi x_2 = 0$, then $\phi = 0$.  We have
\begin{equation}
\label{ }
R_{10}=G_3'\oplus G_2'\oplus G_2'\oplus G_0' \ \text{ and }\  R_{11} = G_3\oplus G_2 \oplus G_1 \oplus G_0,
\end{equation}
it is easy to see that multiplication by $R_1$ maps $G_i'$ to $G_i$, that multiplication by $x_1$ is injective for $i=2,3$, and that the same holds for multiplication by $x_2$.  Moreover multiplication by either $x_1$ or $x_2$ is surjective for $i=0,1$ and the intersection of their kernels is zero.  Writing $\phi = \phi_3 + \dots + \phi_0$ and applying this information we see that $\phi x_1 =\phi x_2 = 0$ implies $\phi = 0$.

Combining these two facts, we see that $d_{f_0}\widetilde{F}$ has rank $28$.  Since its image is an integral element of the holomorphic contact structure $T_hD$, its dimension can be at most half of $56$, the fiber dimension of $T_hD$.  Therefore $\widetilde{F}$ has the highest possible rank of a horizontal holomorphic map, namely $28$.

The second part is easily verified using the above bases of monomials. For $v=x_1$ or $x_2$, both assertions are clear, and they are easily checked for linear combinations $v = a x_1  + b x_2$.

\end{proof}

\subsection{A closed horizontal subvariety of maximum dimension}

Consider now the horizontal holomorphic map $F:S_{10}^0 \to \Gamma\backslash D$.  Following Griffiths (see \S 9 of \cite{griffiths}) we can embed $S_{10}^0 \subset S'$, where $S'$  is a smooth complex manifold containing $S_{10}^0$ as the complement of an analytic subset. One does this by  compactifying with normal crossing divisors.  One can then  extend over the branches of the compactifying divisor for which the monodromy is finite to obtain a proper horizontal holomorphic map $F:S'\to\Gamma\backslash D$.  Then  $F(S')$ is a closed analytic subvariety of $\Gamma\backslash D$ containing $F(S_{10}^0)$ as the complement  of an analytic subvariety.

At the point $f_0\in S_{10}^0$, we found that a local lift $\widetilde{F}:U\to D$ has maximum rank $28$. Consequently, there is a neighborhood $U'$ of $f_0$, where $U'\subset U$,  $\widetilde{F}$ has rank $28$,   and $\widetilde{F}|_{U'}$ is a submersion onto its image.  Therefore  $\widetilde{F}(U')$ is a $28$-dimensional horizontal submanifold of $D$ containing $\widetilde{F}(f_0)$.

We now examine the local structure of $\Gamma\backslash D$.  Since $f_0$ has symmetries, $\widetilde{F}(f_0)$ is fixed by some element $\gamma\in \Gamma, \gamma\ne id$.  Let $\Gamma_0$ denote the subgroup of $\Gamma$ fixing  $\widetilde{F}(f_0)$.  It is necessarily a finite group. If $N$ is a $\Gamma_0$-invariant neighborhood of $\widetilde{F}(f_0)$, then $\Gamma_0\backslash N$ is an orbifold  neighborhood of $F(f_0)$ in the orbifold $\Gamma\backslash D$, and $F(f_0)$ is a singular point of this orbifold.  Strictly speaking, we do not have a tangent space at $F(f_0)$.  But we can move away from $f_0$ in the above neighborhood $U'$ to find non-singular points:

\begin{lemma}
\label{lem:generic}
Let $W\subset (T_h)_{\widetilde{F}(f_0)} D$ denote the image of $d_{(f_0)}\widetilde{F}$.  Then 
\begin{enumerate}
  \item $W$ is not fixed by any $\gamma\in\Gamma_0$, $\gamma\ne id$.
  \item $W$ is not tangent to any horizontal geodesic embedding of  \\$SU(2,14)/S(U(2)\times U(14))$ passing through $\widetilde{F}(f_0)$. 
  \end{enumerate}
\end{lemma}

\begin{proof}
As usual, identify $T_hD$ with $Hom(H^{2,0},H^{1,1})$, and let $V = H^{2,0}$,   $V' = H^{1,1}$.  The group $\Gamma_0$ acts on $T_h D$ through the action of the isotropy group $U(2)\times SO(28)$ of $\widetilde{F}(f_0)$.  Namely $(A,B)$, where $A\in U(2)$ and $B\in SO(28)$ acts on $X\in Hom(V,V')$ by $X\to BXA^{-1}$.

Let us prove the stronger statement that $W$ is not fixed by any element of $U(2)\times SO(28)$:  Suppose $X$ is fixed by $(A,B)\ne id$, say $A\ne id$.  Then $BX = XA$.   Let $\lambda_1,\lambda_2$ be the eigenvalues of $A$ (roots  of unity), and assume, first, that $\lambda_1\ne \lambda_2$ and neither eigenvalue is real.  Let  $V_1,V_2$ be the corresponding eigenspaces, it is easy to see that, for $v_i\in V_i$, $Xv_i$ is an eigenvector for $B$ with eigenvalue $\lambda_i$.  From this we see that $V' = V_1'\oplus V_2'\oplus V_3'$, where $V_1',V_2'$ are the eigenspaces of $B$ for $\lambda_1,\lambda_2$ respectively,  and $V_3'$ is their orthogonal complement. If $X\in W$, then $X(V_i)\subset V_i'$ for $i=1,2$.  In other words, $W\subset Hom(V_1,V_1')\oplus Hom(V_2,V_2')$.  Observe that  $\dim V_1', \dim V_2'\le 14$, since $B$ is real and its eigenvalues come in complex conjugate pairs.  Therefore, if $v_1\in V_1$,
$$
\{Xv_1 \ | \ X\in W\} \subset V_1'.
$$
Since $\dim V_1'\le 14$, this contradicts Proposition  \ref{prop:rank}.  The remaining possibilities for $\lambda_1,\lambda_2$ are  handled by similar  arguments.  This proves that $W$ is not fixed by any element of the isotropy group of $\widetilde{F}(f_0)$. The first part of the Lemma is proved.

For the second part, recall from \S \ref{subsec:construction} that the tangent space to a geodesic embedding of the symmetric space of $SU(2,14)$ through the point $V = H^{2,0}$ is determined by a complex structure $J$ on $V' = H^{1,1}$ and is the subspace  of $X\in Hom(V,V')$ satisfying $JX = Xi$, in other words, the fixed point set of the element $(i,J)$ of $U(2)\times SO(28)$, which we have  already excluded. 

\end{proof}

An immediate consequence of this lemma is that $\widetilde{F}(U')$ is not fixed by any $\gamma\in\Gamma_0$, so there exist $f\in U'$ with $F(f)$ a smooth point of $\Gamma\backslash D$.  The same must be true in a neighborhood $U''\subset U'$ of $f$, so $F|_{U''}:U''\to (\Gamma\backslash D )^0$  (the regular points of $\Gamma\backslash D$) and rank of $dF$ must be $28$ on $U''$. 

In summary:
\begin{theorem}
Let $S'$, $F:S'\to \Gamma\backslash D$ and $\widetilde{F}:\widetilde{S_{10}^0} \to D$  be as above.  Then 

\begin{enumerate}
  \item $F$ is a proper horizontal holomorphic map.
  \item There is a proper analytic subvariety $Z\subset S'$ so that, if $S'' = S'\setminus Z$,  then  $F|_{S''}:S'' \to (\Gamma\backslash D )^0$ and $dF$ has rank $28$ on $S''$.
  \item $F(S')$ is a closed horizontal subvariety of $\Gamma\backslash D$ of maximum possible dimension $28$.
  
  \item If $x\in S''$, the tangent space to $F(S')$ at $F(x)$ is not the tangent space to any totally geodesic immersion of the symmetric space of $SU(2,14)$ in $\Gamma\backslash D$.
  
  \item Alternatively, if $x\in \widetilde{S_{10}^0}$ lies in the dense open set where $d_x \widetilde{F}$ has maximum rank $28$,   the image of $d_x \widetilde{F}$ is not the  tangent space to a geodesic embedding of the symmetric space  $SU(2,14)$ in $D$.
 
\end{enumerate}     
\end{theorem}

\section{Geodesic submanifolds and integral elements} 
\label{sec:integralelements}

We close with some remarks on integral elements of contact structures.  The period domains for which the horizontal bundle gives a contact structure are the twistor spaces of the quaternionic-K\"ahler symmetric spaces, also called the Wolf spaces, see \cite{wolf} for their classification.  We briefly discuss two examples from this point of view: our example $D$, associated to the symmetric space $SO(4,28)/S(O(4)\times O(28))$, and another example we call $D'$ associated to quaternionic hyperbolic space.

Whenever the horizontal sub-bundle $T_hD$ of a domain $D$  is a contact structure,  we know that each fiber of $T_h D$ has a symplectic structure, and the integral elements in that fiber are the Lagrangian subspaces of this symplectic structure.  Lagrangian subspaces of a $2g$-dimensional symplectic space are parametrized by $Sp(g)/U(g)$, the compact dual of the Siegel upper half plane of genus $g$.  

If $D = SO(4,28)/U(2)\times SO(28)$ is the domain we have been studying, of dimension $57$, $T_hD$  of dimension $56$, the integral elements in a fiber of $T_hD$ are parametrized by $Sp(28)/U(28)$,  a manifold of complex dimension $(28 \cdot 29)/2 = 406$.  On the other hand, the totally geodesic embeddings of $D_1$, the symmetric space for $SU(2,14)$ through a fixed point in $D$ are parametrized by the choice of complex structure $J$ on the space $H^+$ as in \S \ref{subsec:construction}.  These are in turn parametrized by the space $SO(28)/U(14)$ of dimension $28\cdot 27 - 14^2 = 14\cdot 13  = 182$.  Thus we see that the space of tangents to geodesic embeddings of $SU(2,14)$ is a rather small subset of the space of Lagrangian subspaces.  We therefore expect the generic horizontal map to miss these embeddings.  In a way, this is what made our example possible.

\subsection{The quaternionic hyperbolic space}
\label{subsec:quaternionic}

We conclude with a related problem, which was the motivation for writing this paper.   Consider the period domain $D'$ associated to the quaternionic hyperbolic space, namely
   \begin{eqnarray}
    \label{eq:fibrationquat}
  \begin{array}{ccc}
Sp(1)/U(1) & \longrightarrow & D' = Sp(1,n)/ U(1)\times Sp(n) \\
& & \Big\downarrow {\pi}      \\
& &  Sp(1,n) / Sp(1)\cdot Sp(n)
\end{array}
\end{eqnarray}
We can think of this domain as classifying Hodge structures on $\R^{4n+4}\cong \bH^{n+1}$ with Hodge numbers $2,4n,2$ which are stable under right multiplication by quaternions.  Equivalently, we can think of points in this domain as pairs $L,J$ where $L\subset \bH^{n+1}$ is a positive right-quaternionic line and $J:L\to L$ is a right quaternionic linear complex structure on $L$ orthogonal with respect to the polarizing form $\left< \ , \ \right>$.  Let $L^\perp$ denote the orthogonal complement of $L$ in $\bH^{n=1}$ and $L_\C, L_\C^\perp$ their complexifications.  Then the horizontal tangent space to the domain $D'$ is 
\begin{equation}
\label{ }
T_n D' = _\C\Hom_\bH (L^{1,0},L_\C^\perp)\subset TD' = _\C\Hom_\bH (L^{1,0},L_\C^\perp\oplus L^{0,1})  \nonumber
\end{equation} 
where $_\C\Hom_\bH$ denotes left $\C$-linear and right $\bH$-linear homomorphisms.  See \S 6 of \cite{carlsontoledo} for a more detailed discussion.

Once again, $D'$ has complex dimension $2n+1$ and $T_h D'$ has fiber dimension $2n$, so it is a holomorphic contact structure on $D'$.  
Each fiber  of $T_hD'$ has a symplectic structure, and the integral elements of the contact structure in a fixed fiber coincide with the Lagrangians of this symplectic structure, and are therefore parametrized by $Sp(n)/U(n)$.

We also have horizontal totally geodesic embeddings of the symmetric space of $SU(1,n)$ in $D'$, namely the unit ball or complex hyperbolic space $SU(1,n)/U(n)$.  The group $Sp(n)$ acts transitively on the  embeddings passing through a point $(L,J)$, corresponding to orthogonal right $\bH$-linear  complex structures on $L^\perp$, hence parametrized by the same homogeneous space $Sp(n)/U(n)$ that parametrizes the Lagrangians.  Thus, for $D'$, every horizontal subvariety of maximum dimension $n$ is tangent, at each smooth point, to a horizontal totally geodesic  complex hyperbolic $n$-space.  (We used this fact in \S 6 of \cite{carlsontoledo} to give a structure theory for harmonic maps of K\"ahler manifolds to manifolds covered by quaternionic hyperbolic space).  

\begin{problem}
Find examples of discrete groups $\Gamma\subset Sp(1,n)$ and closed horizontal subvarieties $V\subset \Gamma\backslash D'$ that are not totally geodesic.
\end{problem}

\end{document}